\documentclass[11pt]{amsart}
\usepackage{hyperref}
\usepackage{amsfonts,mathrsfs,bbm,latexsym,rawfonts,amsmath,amssymb,amsthm,epsfig}
\usepackage{fullpage, setspace, color}

\newtheorem{theorem}{Theorem}[section]
\newtheorem{lemma}[theorem]{Lemma}

\newtheorem{thm}{Theorem}[section]

\newtheorem{rem}[thm]{Remark}
\newtheorem{lem}[thm]{Lemma}

\newtheorem{defn}{Definition}[section]

\numberwithin{equation}{section}

\newcommand{\al}{\alpha}

\newcommand{\ld}{\lambda}

\newcommand{\de}{\delta}
\newcommand{\De}{\Delta}
\newcommand{\ep}{\varepsilon}

\newcommand{\Om}{\Omega}
\newcommand{\ga}{\gamma}
\newcommand{\Ga}{\Gamma}

\renewcommand{\th}{\theta}

\newcommand{\J}{\mathfrak{J}}

\newcommand{\U}{\mathbb{S}}

\newcommand{\Real}{\mathbb{R}}

\newcommand{\norm}[1]{\Vert#1\Vert}

\def\<{\left\langle} \def\>{\right\rangle}
\def\({\left(} \def\){\right)}
\newcommand{\n}{\nabla}
\newcommand{\p}{\partial}

\begin{document}
	
\title{Global weak solutions for Landau-Lifshitz flows and heat flows associated to micromagnetic energy functional}
\thanks{*Corresponding Author}
\author{Bo Chen}
\address{University of Chinese Academy of Sciences, Chinese Academy of Sciences, Beijing,100190, P.R.China}
\email{chenbo@amss.ac.cn}

\author{Youde Wang*}
\address{1. College of Mathematics and Information Sciences, Guangzhou University;
2. Hua Loo-Keng Key Laboratory of Mathematics, Institute of Mathematics, AMSS, and School of
Mathematical Sciences, UCAS, Beijing 100190, China.}
\email{wyd@math.ac.cn}
\date{\today}
\begin{abstract}
We follow the idea of Wang \cite{W} to show the existence of global weak solutions to the Cauchy problems of Landau-Lifshtiz type equations and related heat flows from a $n$-dimensional Euclidean domain $\Om$ or a $n$-dimensional closed Riemannian manifold $M$ into a 2-dimensional unit sphere $\U^{2}$. Our conclusions extend a series of related results obtained in the previous literature.
\end{abstract}
\maketitle
\section{Introduction}
In physics, the Landau-Lifshtiz (LL) equation is a fundamental evolution equation for the ferromagnetic spin chain and was proposed on the phenomenological ground in studying the dispersive theory of magnetization of ferromagnets. It was first deduced by Landau and Lifshitz in \cite{LL}, and then proposed by Gilbert in \cite{G} with dissipation. In fact, this equation describes the Hamiltonian dynamics corresponding to the Landau-Lifshitz energy, which is defined as follows.

Let $\Omega$ be a smooth bounded domain in the Euclidean space $\mathbb{R}^3$. The generic point of $\mathbb{R}^3$ is denoted by $x$. We assume that a ferromagnetic material occupies the domain $\Omega\subset\mathbb{R}^3$. Let $u$, denoting magnetization vector, be a mapping from $\Omega$ into a unit sphere $\U^2\subset\mathbb{R}^3$. The Landau-Lifshitz energy of map $u$ is defined by
$$\mathcal{E}(u):=\int_{\Om}\Phi(u)\,dx+\frac{1}{2}\int_{\Om}|\n u|^2\,dx-\frac{1}{2}\int_{\Omega}h_d\cdot u\,dx.$$
Here the $\n $ denote the gradient operator and $dx$ is the volume element of $\mathbb{R}^3$.

In the above Landau-Lifshitz functional, the first and second terms are the anisotropy and exchange energies, respectively. $\Phi(u)$ is a real function on $\U^2$. If one only considers uniaxial materials with easy axis parallel to the OX-axis, for which $\Phi(u)= u_2^2 + u_3^2$. The last term is the self-induced energy, and $h_d = -\nabla w$ is the demagnetizing field. The magnetostatic potential $w$ solves the differential equation
$$\Delta w = \mbox{div}(u\chi_{\Omega})$$
in $\mathbb{R}^3$ in the sense of distributions, where $\chi_{\Omega}$ is the characteristic function of $\Om$ and $\De$ is the Laplace operator on $\mathbb{R}^3$.

More precisely, the solution to the Poisson equation is
$$w(x) = \int_{\Omega} \nabla N(x-y)u(y)dy,$$
where $N(x) = -\frac{1}{4\pi|x|}$ is the Newtonian potential in $\mathbb{R}^3$.

The LL equation with dissipation, which can be written as
\[
u_t - \alpha u\times u_t = u\times h,\]
where ``$\times$" denotes the cross production in $\Real^{3}$ and the local field $h$ of $\mathcal{E}(u)$ can be derived as
$$h:=-\frac{\delta\mathcal{E}(u)}{\delta u}= \Delta u + h_d -\nabla_u\Phi.$$
Meanwhile the constant $\alpha$ is the damping parameter, which is characteristic of the material, and is usually called the Gilbert damping coefficient.  Hence the Landau-Lifshitz equation with damping term is also called the Landau-Lifshitz-Gilbert (LLG) equation in the literature.

Moreover, the interplay of spin-polarized electrical currents and local magnetic moments revealed in the original theoretical studies of Berger \cite{Be} and Slonczewski \cite{S} has stimulated a great deal of research effort in the nanoscale magnetic structures.
In this paper we are interested in a mathematical model describing magnetization dynamics by spin-polarized current.  With a prescribed current density $J (x, t)$, the time evolution of the magnetization vector $u(x,t)$ may be described by the LLG equation, see for example \cite{KTS},
\begin{equation}\label{cur}
\left\{
\begin{array}{llll}
\aligned
&u_t-\alpha u\times u_t=-\gamma u\times[h +\beta(J\cdot\n u)]\quad\mbox{in}\,\, \Omega\times(0,T),\\
&u(0,\cdot)=u_0:\,\Omega\rightarrow \U^2,\\
&\frac{\partial u}{\partial\nu}=0\quad\mbox{on}\,\, \partial\Omega,
\endaligned
\end{array}
\right.
\end{equation}
where $T$ is any fixed positive number  and $\nu$ represents the outward unit normal on $\p\Om$.
The term parameterized by the positive constant $\beta$ expresses current-induced torques on $u$. This torque is most commonly termed non-adiabatic and $\beta$ characterizes its strength. The parameter $\gamma>0$ is a gyroscopic ratio. The above initial value map
satisfied by the magnetization is
$$u(x,0) = u_0(x) \quad\quad \mbox{and} \quad\quad |u_0(x)|=1 \,\,\quad a.e.\,\, \mbox{in}\,\,\, \Omega.$$

First, we note a fact that for $u:\Omega\times\mathbb{R}^+\to \U^2$ the following equation
$$ u_t -\alpha u\times u_t=-\gamma u\times (h+\beta J\cdot\n u)$$
is equivalent to
$$u_t= -\frac{\alpha\gamma}{\alpha^2+1}u\times(u\times (h+\beta J\cdot \n u)) - \frac{\ga}{\alpha^2+1}u\times (h+\beta J\cdot \n u).$$
In the sequel, we always assume that $\alpha>0$, $\ga=1+\al^2$. The previous equation can be written as
\begin{equation}\label{cur1}
\left\{
\begin{array}{llll}
\aligned
&u_t= -\al u\times(u\times h_\beta)-u\times h_\beta\quad\quad\mbox{in}\,\, \Omega\times(0,\,T),\\
&u(\cdot, 0)=u_0:\,\,\,\Omega\rightarrow \U^2,\quad\quad \frac{\partial u}{\partial\nu}=0\quad\quad \mbox{on}\,\, \partial\Omega\times(0,\,T),
\endaligned
\end{array}
\right.
\end{equation}
where $$h_\beta\equiv h +\beta J\cdot\n u=\Delta u + h_d -\nabla_u\Phi+\beta J\cdot\n u.$$

\medskip

On the other hand, we are also interested in  the heat flow associated to micromagnetic energy functional with spin-polarized current as follows.
\begin{equation}\label{hf}
\left\{
\begin{aligned}
& u_t=\tau(u)+\(\tilde{h}_\beta(u)-\<\tilde{h}_\beta(u),u\>u\), \quad (x, t)\in \Om\times(0,T),\\
& u(0)=u_0: \Om\to \U^{2},\quad\quad\frac{\partial u}{\partial\nu}=0,\quad \mbox{on}\,\, \p\Om\times [0,T],
\end{aligned}\right.
\end{equation}
where $\tau(u)=\De u+|\n u|^{2}u$ is the tension field, and $\tilde{h}_\beta(u)=h_d -\nabla_u\Phi+\beta J\cdot \n u$. Since $|u|=1$, it can be rewritten in the form
\[\label{hf1}
\begin{cases}
 u_t=-u\times(u\times h_{\beta}),   &(x, t)\in \Om\times(0,T),\\
 u(0)=u_0: \Om\to \U^{2},           &\mbox{on}\,\,\Omega,\\
 \displaystyle\frac{\partial u}{\partial\nu}=0,  &\mbox{on}\,\, \p\Om\times[0,T].
\end{cases}
\]

Generally, let $\Om$ be a bounded smoothly domain in $\Real^{m}$ for $m\geq 3$, we extend the micromagnetic energy functional without the self-induced energy as follows
$$\bar{\mathcal{E}}(u):=\int_{\Omega}\Phi(u)\,dx+\frac{1}{2}\int_{\Omega}|\nabla u|^2\,dx$$
where $u:\Om\to \U^{2}$ is a map and $\Phi: \U^{2}\to \Real^{+}$ is a smooth function. The critical point $u$ of $\bar{\mathcal{E}}$ satisfies the following Euler-Lagrangian equation
$$\tau_\Phi(u)=\tau(u)-\(\n_u\Phi(u)-\<\n_u\Phi(u),u\>u\)=0,$$
this is just the harmonic map with potential. Thus, the following heat flow associated to this energy
\begin{equation}\label{hf2}
\left\{
\begin{aligned}
& u_t=\tau_\Phi(u)+\beta J\cdot \n u,\quad (x, t)\in \Om\times(0,T),\\
& u(0)=u_0: \Om\to \U^{2},\quad\frac{\partial u}{\partial\nu}=0\quad \mbox{on}\,\, \partial\Omega\times [0,T],
\end{aligned}\right.
\end{equation}
is also of significance in mathematical aspect, where $J:\Om\times \Real^{+}\to \Real^{m}$ is a smooth function. Using the property of the cross-product in $\mathbb{R}^3$, it is easy to see that the above equation is equivalent to
\begin{equation}\label{hf-m}
\left\{
\begin{aligned}
& u_t=-u\times(u\times \bar{h}_\beta), \quad (x, t)\in \Om\times(0,T),\\
& u(0)=u_0: \Om\to \U^{2},\quad \frac{\partial u}{\partial\nu}=0\quad \mbox{on}\,\, \partial\Omega\times [0,T].
\end{aligned}\right.
\end{equation}
Here, $$\bar{h}_{\beta}= \Delta u -\n_u\Phi+\beta J\cdot \n u.$$

In fact, $$\frac{DU}{Dt}\equiv u_t-\beta J\cdot \n u$$ can be regarded as a material derivative of $u$ and the equation (\ref{hf-m}) appears in the liquid crystal theory. Indeed, the simplified Ericksen-Leslie system that models the hydrodynamics of nematic liquid crystals in dimension three: for a bounded smooth domain $\Omega\subset\mathbb{R}^3$ (or $\Omega = \mathbb{R}^3$ ) and $0 < T \leq \infty$, $(w, P, u): \Omega\times(0, T)\to\mathbb{R}^3 \times\mathbb{R}\times \mathbb{S}^2$ solves
\begin{equation*}
\left\{
\begin{aligned}
& w_t+w\cdot\nabla w-\nu\Delta w +\nabla P =-\lambda\nabla\cdot(\nabla u\odot\nabla u), &(x, t)\in \Om\times(0,T),\\
& \nabla\cdot w=0, &(x, t)\in \Om\times(0,T),\\
& u_t+w\cdot\nabla u = \alpha(\Delta u + |\nabla u|^2u), &(x, t)\in \Om\times(0,T),
\end{aligned}\right.
\end{equation*}
along with the initial and boundary condition
\[
\begin{cases}
(w(0), u(0))=(w_0, u_0), &(x, t)\in \Om,\\
(w,u)=(0, u_0), &(x, t)\in \partial\Om\times(0,\infty),
\end{cases}
\]
for a given initial datum $(w_0, u_0):\Omega\times(0, T)\to\mathbb{R}^3\times\mathbb{S}^2$ with $\nabla\cdot w_0 = 0$. Here $w:\Om\to\mathbb{R}^3$ represents the velocity field of the fluid, $u: \Om\to\mathbb{S}^2$ is a unit vector field representing the macroscopic orientation of the nematic liquid crystal molecules, and $P: \Om\to \mathbb{R}$ represents the pressure function.  The constants $\alpha$, $\gamma$ and $\nu$ are positive constants. From the mathematical point of view, this is a system strongly coupling the transported heat flow of harmonic maps to $\mathbb{S}^2$ and the nonhomogeneous incompressible Navier-Stokes equation. For more details, we refer to \cite{L}.

Another natural generalization is to consider the above flow in the setting of closed Riemannian manifold. In fact, let $(M,g)$ be a closed Riemannian manifold with dimension $m\geq 3$, we consider the following flow
\begin{equation}\label{hf-m1}
\left\{
\begin{aligned}
& u_t=-u\times(u\times \bar{h}_\beta), \quad (x, t)\in M\times\mathbb{R}^+,\\
& u(0)=u_0: M \to \U^{2},
\end{aligned}\right.
\end{equation}
where $J: M\to TM$ be a section of the tangent bundle of $M$.

\medskip
In recent years, there has been lots of interesting studies for the Landau-Lifshitz equation, concerning its existence, uniqueness and regularities of various kinds of solutions. Before moving on to the next step, we list only a few of the literature that are closely related to our work in the present paper.

For the case $\Om$ is a bounded domain in $\mathbb{R}^3$, Carbou and Fabrie studied a model of ferromagnetic material governed by a nonlinear dissipative Landau-Lifschitz equation (i.e. $\alpha>0$) coupled with Maxwell equations in micromagnetism theory, and they proved the local existence and uniqueness of regular solutions for a so-called quasistatic model in \cite{CF}. Moreover they showed global existence of regular solutions for small data in the 2D case for the Landau-Lifschitz equation. Later, Tilioua \cite{T} (also see \cite{Bo}) employed the penalized method to show the existence of the weak solution to \eqref{cur}, without the anisotropy term $\Phi$, in the case damping constant $\alpha>0$. Recently, in \cite{CJ} the local existence of very regular solutions to (\ref{cur}) was addressed.

Next, we retrospect the work to related flows. The heat flows associated to Landau-Lifschitz functional, defined by \eqref{hf} and \eqref{hf2} in above, can be considered as generation of the flows for harmonic maps into $\U^{2}$. The Global weak solutions of latter flows have been well-researched by Y.M. Chen and el in \cite{ChenYM1, ChenYM2}, by also using the classical Ginzburg-Landau penalized method. Very recently, we achieved a blow-up result of finite time for the flow related with the micromagnetic energy energy in \cite{Chen-Wang}, which is well-known for the heat flows of harmonic map very early.

However, we should mention that the penalized method, used in \cite{T, ChenYM1, ChenYM2} to get global weak solution, may not be effective for equation
      $$\partial_tu-\alpha u\times\partial_t u = -u\times(\Delta u+h_d-\nabla_u\Phi + \beta J\cdot\n u),$$
since the function $\Phi$ is defined on $\U^{2}$, which doesn't match well with their approximated equations. To deal with this term, we would extend $\Phi$ to a function $\tilde{\Phi}$ defined on $\bar{B}_{1}(0)$ to compatible with our approximated equation in Section $3$. Secondly, such methods applied in \cite{T, Bo} is only valid for LLG equations or heat flows related to them in the case that the dimension of the domain space is 3 or 4. More precisely, if $\Om$ be a domain in $\Real^m$ with $m\geq 3$, then, to employ the penalized method one needs to suppose the initial map $u_{0}\in L^{4}(\Om,\Real^{3})$. However, the suitable assumption on $u_{0}$ to get weak solution is that $u_0\in W^{1,2}$. So, it is necessary to embed $W^{1,2}$ to $L^{4}$, under the restriction of dimension of $\Om$
$$\frac{2m}{m-2}\leq 4,$$
that is $m\leq 4$.

In order to overcome the above obstruction, in this paper we follow the idea in \cite{W} to approach the existence problems of the equations with spin currents in general dimension case. One of the crucial ingredients for the presented analysis here is the choices of effective auxiliary approximation equations and test functions. It is the aim of the paper at hand to present a proof of the existence of global weak solution (the definitions are given in Section 2). Our main conclusions can be presented as follows

\begin{thm}\label{thm1}
Let $\Omega$ be a bounded smooth domain in $\mathbb{R}^3$. Assume that $\Phi\in C^{1}(\U^2, \mathbb{R})$, the initial value map $u_0\in H^{1}(\Omega, \U^2)$ and $J\in L^{2}(\Real^+,L^{\infty}(\Om,\Real^{3}))$ is a measurable function of vector value. Then, for any $\alpha > 0$ and any $T>0$, the equation \eqref{cur1} admits a global weak solution $u\in L^{\infty}([0,T],H^{1}(\Om,\U^{2}))\cap W^{1,1}_{2}(\Om\times[0,T],\U^{2})$ with initial value $u_0$.
\end{thm}

\begin{thm}\label{thm2}
Let $\Om$ be a bounded smooth domain in $\Real^{m}$ with $m\geq 3$.  Assume that $\Phi\in C^{1}(\U^2, \mathbb{R})$, the initial value maps $u_0\in H^{1}(\Om, \U^2)$ and $J\in L^{2}(\Real^+,L^{\infty}(\Om,\Real^{3}))$ is a measurable function of vector value. Then, for any $T>0$, the heat flow \eqref{hf-m} admits a global weak solution $u\in L^{\infty}([0,T],H^{1}(\Om,\U^{2}))\cap W^{1,1}_{2}(\Om\times[0,T],\U^{2})$ with initial value $u_0$.
\end{thm}

As a corollary, we have
\begin{thm}\label{thm3}
Let $(M,g)$ be a closed $m$-dimensional Riemannian manifold, $m\geq 3$.  Assume that $\Phi\in C^{1}(\U^2, \mathbb{R})$, the initial value maps $u_0\in H^{1}(M, \U^2)$ and $J\in L^{2}(\Real^+,L^{\infty}(M,TM))$ is a measurable function of vector value. Then,  the heat flow \eqref{hf-m1} admits a global weak solution $u\in L^{\infty}([0,T],H^{1}(M,\U^{2}))\cap W^{1,1}_{2}(M\times[0,T],\U^{2})$ with initial value $u_0$.
\end{thm}

The rest of our paper is organized as follows. In section \ref{s: pre}, we introduce the basic notations on Riemannian and some critical preliminary lemmas. Meanwhile the definitions of weak solutions to equations \eqref{cur1}\eqref{hf-m}\eqref{hf-m1} will also be given.
In section \ref{s: LL-c}, we give the proof of Theorem \ref{thm1}. Finally, the proof of Theorem \ref{thm2} and \ref{thm3} will built up in section \ref{s: hf-m}.

\section{Preliminary}\label{s: pre}
\subsection{Notations on Riemannian manifold}
In this section, we recall some notations on manifolds. Let $(M, g)$ and $(N,\tilde{g})$ be two Riemannian manifolds and $N$ be embedded isometrically in $\Real^K$. The Dirichlet-Schmidt energy of a smooth map $u=(u_1,\dots,u_K)$ from $M$ into $N\subset \Real^{K}$ is defined by
$$E(u)=\frac{1}{2}\int_M|\nabla u|^2\,d\mu_{g},$$
where $d\mu_{g}$ is the volume element induced on $M$. The tension field of a map $u$  is given by
 $$\tau(u)=\Delta_g u +A(u)(\n u, \n u),$$
 where $\Delta_g$ denotes the Laplace-Beltrami operator on $(M, g)$ and $A(\cdot,\cdot)$ is the second fundamental form of the embedding $(N,\tilde{g})\subset\Real^K$. More precisely, in local charts $(x^{1}, \dots,x^{n})$ of $M$, it can be written as
 $$\tau(u)^\th=\Delta_g u_\th + g^{ij}(x)\Ga^\th_{\beta\ga}(u)\frac{\p u_\beta}{\p x_i}\frac{\p u_\ga}{\p x_j}.$$
 Here, $\Ga^\al_{\beta\ga}$ is the Christoffel symbols of the Riemannian connection on $(N, \tilde{g})$ and
 \[
 \De_g u_\th=\frac{1}{\sqrt{g}}\frac{\partial}{\partial x^i}\big(g^{ij}\sqrt{g}\frac{\partial u_\th}{\partial x^j} \big),\,\,\,\,\quad\mbox{for}\quad\th=1,2,...,K.
 \]
 For convenience we always denote $\De_g$ by $\De$. In particular, if $(N,\tilde{g})=(\U^{2},g_{\U^{2}})$, where $g_{\U^{2}}$ is the metric induced by Euclidean metric on $\Real^{3}$, we have
 $$\tau(u)=\De u+|\n u|^{2}u.$$

 On the other hand, for a smooth function $J: M\to TM$, we denote
 \[
 J\cdot\nabla u=(J\cdot\nabla u_1,\,...\, , J\cdot\nabla u_m)
 \]
 where
 \[
 J\cdot\nabla u_\th=J_{i}\frac{\partial u_\th}{\partial x^p}g^{ij},\,\,\,\,\quad\mbox{for}\quad \th=1,2,...,K.
 \]
 \medskip

Now, we define the Sobolev spaces of the functions from $M$ to $N$ as follows
 $$H^k(M, N)=\{u\in H^k(M, \mathbb{R}^K):\, u(x)\in N \,\text{a.e.}\,\, x\in M\}.$$
Moreover, we define
 $$W^{k,l}_p(M\times [0,T], N)=\{u\in W^{r,s}_p(M\times[0,T], \mathbb{R}^K): \,u(x)\in N \,\text{a.e.}\,\, x\in M\},$$
where $k,\,l\in \mathbb{N}$  and $p\geq 1$, $T>0$.
\subsection{Weak solution}
Now, we need to give the definition of the weak solutions.
\begin{defn}
	Let $\Om$ be a bounded smooth domain in $\Real^{3}$, $u\in L^{\infty}([0,T],\, H^1(\Omega,\U^2))\cap W^{1,1}_2(\Omega\times[0,T],\,\U^2)$ is called the weak solutions to equation \eqref{cur1} with initial values $u_0$ if $u$ satisfies
	\[
	\aligned
	&\int_{0}^{T}\int_{\Om}(\< u_t,\varphi\>-\al\< u\times u_t,\varphi\>)\,dxdt \\
	=&(1+\al^{2})\int_{0}^{T}\int_{\Om}\(\< u\times\n u,\nabla\varphi\> -\<u\times(-\nabla_u\Phi+h_d+\beta J\cdot\n u),\, \varphi\>\)\,dxdt,
	\endaligned
	\]
	for all $\varphi\in C^{\infty}(\bar{\Om}\times[0,T],\,\U^2)$, where $\bar{\Om}$ is the closure of $\Om$ in $\Real^{3}$.
	\medskip	
\end{defn}

\medskip
 In the case of arbitrary dimensions, we have the following definitions of weak solutions to heat flows \eqref{hf-m} and \eqref{hf-m1}.
 \begin{defn}
 Let $\Om$ be a bounded smooth domain in $\Real^{m}$ with $m\geq 3$. We say $u\in L^{\infty}([0,T],\, H^1(\Om,\U^2))\cap W^{1,1}_2(\Om\times[0,T],\,\U^2)$ is a global weak solution of heat flow \eqref{hf-m} with initial data $u_0$ if
 	\begin{eqnarray*}
 		&&\int_0^T\int_{\Omega}\<\partial_tu,\varphi\>dxdt+\int_0^T\int_{\Om}\<\n u,\n \varphi\>dxdt\\
 		&=&\int_0^T\int_{\Om}|\n u|^{2}\<u,\varphi\>dxdt+\int_0^T\int_{\Om}\<u\times(-\n_{u}\Phi+\beta J\cdot \n u),u\times\varphi\>dxdt
 	\end{eqnarray*}
 	for all $\varphi\in C^{\infty}(\bar{\Om}\times[0,T],\U^{2})$.
 \end{defn}
\begin{defn}
	Let $(M,g)$ be a $m$-dimensional Riemannian manifold with $m\geq 3$. We say $u\in L^{\infty}([0,T],\, H^1(M,\U^2))\cap W^{1,1}_2(M\times[0,T],\,\U^2)$ is a global weak solution of heat flow \eqref{hf-m1} with initial data $u_0$ if
	\begin{eqnarray*}
		&&\int_0^T\int_{M}\<\partial_tu,\varphi\>d\mu_g dt+\int_0^T\int_{M}\<\n u,\n \varphi\>d\mu_gdt\\
		&=&\int_0^T\int_{M}|\n u|^{2}\<u,\varphi\>d\mu_g dt+\int_0^T\int_{M}\<u\times(-\n_{u}\Phi+\beta J\cdot \n u),u\times\varphi\>d\mu_g dt
	\end{eqnarray*}
	for all $\varphi\in C^{\infty}(M\times[0,T],\U^{2})$.
\end{defn}
\subsection{Some analysis results}\

For later application, we firstly introduce some regular results. The following estimate of demagnetizing field $h_d$ and the lemma about equivalent norm for Sobolev function with Neumann boundary condition can be found in \cite{CF}\cite{CJ}\cite{Chen-Wang}.
\begin{lemma}\label{es-h_d}
	Let $p\in(1,\infty)$, $\Om$ be a bounded smooth domain in $\Real^{3}$. Assume that $u \in W^{k,p}(\Om,\Real^{3})$ for $k\in \mathbb{N}$, then the restriction of $h_d(u)$ to $\Om$ belongs to $W^{k,p}(\Om, \Real^{3})$. Moreover, there exists constants $C_{k,p}$ independent of $u$, such that
	$$\lVert h_d(u)\rVert_{W^{k,p}(\Om)}\leq C_{k,p}\lVert u\rVert_{W^{k,p}(\Om)}.$$
Moreover, $h_{d}:W^{k,p}(\Om, \Real^{3})\to W^{k,p}(\Om, \Real^{3})$ is a linear bounded operator.
\end{lemma}
\begin{lemma}\label{eq-norm}
Let $\Om$ be a bounded smooth domain in $\Real^{m}$ and $k\in \mathbb{N}$. There exists a constant $C$ such that for all $u\in H^{k+2}(\Om)$ with $\frac{\p u}{\p \nu}|_{\Om}=0$,
\begin{equation}\label{eq-n}
\norm{u}_{H^{2+k}(\Om)}\leq C(\norm{u}_{L^{2}(\Om)}+\norm{\De u}_{H^{k}(\Om)}).
\end{equation}	
\end{lemma}
In particular, we define the the $H^2$-norm of $u$ as follows
$$\norm{u}_{H^{2}(\Om)}:=\norm{u}_{L^2(\Om)}+\norm{\De u}_{L^2(\Om)}.$$
\begin{rem}
	If let $(M,g)$ be a closed Riemannian manifold, the above inequality \eqref{eq-n} also holds for any $u\in H^{k+2}(M)$.
\end{rem}
 Finally, we give a classical compact result in \cite{Sim}, which will be used to get the convergence of solutions to the approximated equation constructed in the next section.

 \begin{lemma}[Aubin-Simon compact lemma]\label{A-S}
 	Let $X\subset B\subset Y$ be Banach spaces with compact embedding $X\hookrightarrow B$, $F$ is a bounded set in $L^{q}([0,T],B)$ for $q>1$.
 	If $F$ is bounded in $L^{1}([0,T],X)$ and $\frac{\p F}{\p t}$ is bounded in $L^{1}([0,T],Y)$, then $F$ is relatively compact in $L^{p}([0,T],B)$, for any $1\leq p<q$.
 \end{lemma}

\section{Landau-Lifshitz Equation with Spin Currents}\label{s: LL-c}

In this section we consider the global well-posedness of \eqref{cur1}. For this end we adopt the following approximate equation
\begin{equation}\label{app-cur}
\left\{
\begin{array}{llll}
\aligned
&u_t=\varepsilon\Delta u-\alpha\mathfrak{J}(u)\times(\mathfrak{J}(u)\times h_\beta(u)) - \mathfrak{J}(u)\times h_\beta(u) \quad \mbox{in}\,\, \Omega\times(0,T),\\
&u(0,\cdot)=u_0:\,\Omega\rightarrow S^2,\quad\quad \frac{\partial u}{\partial\nu}=0\quad \mbox{on}\,\, \partial\Om\times[0,T],
\endaligned
\end{array}
\right.
\end{equation}
where $0<\varepsilon<1$ is a positive constant, $\mathfrak{J}(u)$ is defined by
$$\mathfrak{J}(u)\equiv\frac{u}{\max\{1,|u|\}}$$
and we still denote $h_{\beta}(u)=\De u-\n_u\Phi(\J(u))+h_d(u_\ep)+\beta J\cdot\n u$.

It should also be pointed out that the above $\Phi(u)$ has been extended to the closed ball $\overline{B}(1)\subset\mathbb{R}^3$. In fact, we extend $\Phi(z)$ by
\begin{equation*}
\tilde{\Phi}(z)=
\left\{
\begin{aligned}
&\zeta(|z|^2)\Phi\left(\frac{z}{\max\{\delta_0,|z|\}}\right),\,\, &|z|^2>\delta_0,\\
&0, \,\, &|z|^2\leq\delta_0,
\end{aligned}\right.
\end{equation*}
where $\zeta(t):[0, 1]\to [0, 1]$ is a $C^2$-smooth function with $\zeta(t)\equiv 0$ on $[0, \, 2\delta_0]$ ($2\delta_0<1$) and $\zeta(1)=1$. It is easy to see that $\tilde{\Phi}$ is $C^1$-smooth on $\overline{B}(1)$. For simplicity, we still denote $\tilde{\Phi}$ by $\Phi$.

Next, we will construct a weak solution of \eqref{app-cur} by the classical Galerkin Approximation Method and obtain quantitative energy estimate on its solutions.

\subsection{Galerkin Approximation and a prior estimates}
Let $\Om$ be a bounded smooth domain in $\Real^{3}$, $\ld_{i}$ be the $i^{th}$ eigenvalue of the operator $\De-I$  with Neumann boundary condition, whose corresponding eigenfunction is $f_{i}$. That is,
$$(\De-I)f_{i}=-\ld_{i}f_{i}\,\,\,\,\quad\text{with}\quad\,\,\,\,\frac{\p f_{i}}{\p \nu}|_{\p\Om}=0.$$

Without loss of generality, we assume $\{f_{i}\}_{i=1}^{\infty}$ are completely, standard orthonormal basis of $L^{2}(\Om,\Real^{1})$. Let $H_{n}=span\{f_{1},\dots f_{n}\}$ be a finite subspace of $L^{2}$, $P_{n}:L^{2}\to H_{n}$ be the canonical projection. In fact, for any $f\in L^{2}$, $f^{n}=P_{n}f=\sum_{1}^{n}\<f,f_{i}\>_{L^{2}}f_{i}$ and $\lim_{n\to \infty}\norm{f-f_{n}}_{L^{2}}=0$.

Next, we seek a solution $u^{n}_{\ep}$ in $H_{n}$ to the above Galerkin approximation equation associated to \eqref{app-cur}, i.e.
\begin{equation}\label{app-cur1}
\begin{cases}
\displaystyle\frac{\p u^{n}_{\ep}}{\p t}=\ep\De u^{n}_{\ep}-P_{n}\{\al \J(u^{n}_{\ep})\times(\J(u^{n}_{\ep})\times h_\beta(u^n_\ep))+\J(u^{n}_{\ep})\times h_{\beta}(u^n_\ep)\},&(x,t)\in\Om\times(0,T),\\
u^{n}_{\ep}(x,0)=u^{n}_{0}(x),&x\in\Om.
\end{cases}
\end{equation}
Let $u^{n}_{\ep}=\sum_{1}^{n}g^{n}_{i}(t)f_{i}(x)$, $g^{n}(t)=\{g^{n}_{1}(t)\dots g^{n}_{n}(t)\}$ be a vector-valued function. Then, by a directed calculation we have that $g^{n}(t)$ satisfies the following ordinary differential equation
\begin{equation}\left\{
\begin{aligned}
& \frac{\p g^{n}}{\p t}=F(g^n),\\
& g^{n}(0)=(\<u_{0},f_{1}\>,\dots,\<u_{0},f_{n}\>),
\end{aligned}\right.
\end{equation}
where $F(g^n)$ is locally Lipschitz on $g^n$, since $\J(f)$ is locally Lipschitz on $f$. Hence, there exist a solution $u^{n}_{\ep}$ to \eqref{app-cur1} on $\Om\times[0,T_{0}]$ for some $T_{0}>0$.

Multiplying $u^{n}_\ep$ in both sides of \eqref{app-cur1}, there holds
     $$\frac{\p }{\p t}\int_{\Om}|u^{n}_{\ep}|^{2}dx=-2\ep\int_{\Om}|\n u^{n}_{\ep}|^{2}dx\leq 0,$$
 which implies
     $$\int_{\Om}|u^{n}_{\ep}|^{2}dx(t)=\sum_{1}^{n}\<g^{n}_{i},g^{n}_{i}\>dx(t)\leq \int_{\Om}|u^{n}_{0}|^{2}dx\leq \int_{\Om}|u_{0}|^{2}dx\leq \infty,$$
for any $0<t\leq T_{0}$. Thus, it follows that the above solution $g^{n}$ can be extended as a global one, so as $u^{n}_{\ep}$.

To get the $H^{2}$-energy estimates of $u^{n}_{\ep}$, we choose the test function $v=-\De u^{n}_{\ep}$. A simple computation shows
\begin{equation}\label{ineq}
\begin{aligned}
&\frac{1}{2}\frac{\p }{\p t}\int_{\Om}|\n u^{n}_{\ep}|^{2}dx+\ep\int_{\Om}|\De u^{n}_{\ep}|^{2}dx+\al\int_{\Om}|\J(u^{n}_{\ep})\times \De u^{n}_{\ep}|^{2}dx\\
=&-\al\int_{\Om}\<\J(u^{n}_{\ep})\times \De u^{n}_{\ep},\J(u^{n}_{\ep})\times(\n_u\Phi(\J(u^{n}_{\ep}))+h_{d}(u^{n}_{\ep})+\beta J\cdot\n u^{n}_{\ep})\>dx\\
&-\int_{\Om}\<\J(u^{n}_{\ep})\times\De u^{n}_{\ep},\n_u\Phi(\J(u^{n}_{\ep}))+h_{d}(u^{n}_{\ep})+\beta J\cdot\n u^{n}_{\ep}\>dx,
\end{aligned}
\end{equation}
where we have used the identity $\<Y,X\times Z\>+\<X\times Y,Z\>=0$ for any vector fields $X$, $Y$ and $Z$ in $\Real^{3}$.
By direct calculations, we have
\begin{equation}\label{ineq1}
\begin{aligned}
&\int_{\Om}\<\J(u^{n}_{\ep})\times \De u^{n}_{\ep},\,\J(u^{n}_{\ep})\times\n_u\Phi(\J(u^{n}_{\ep}))\>dx\\
\leq&\de \int_{\Om}|\J(u^{n}_{\ep})\times \De u^{n}_{\ep}|^{2}dx+C(\de)\int_{\Om}|\n_u\Phi(\J(u^{n}_{\ep})|^{2}dx,
\end{aligned}
\end{equation}
\begin{equation}\label{ineq2}
\begin{aligned}
\int_{\Om}\<\J(u^{n}_{\ep})\times \De u^{n}_{\ep},\,\J(u^{n}_{\ep})\times h_{d}(u^{n}_{\ep})\>dx
\leq \de \int_{\Om}|\J(u^{n}_{\ep})\times \De u^{n}_{\ep}|^{2}dx+C(\de)\int_{\Om}|u^{n}_{\ep}|^{2}dx,
\end{aligned}
\end{equation}
\begin{equation}\label{ineq3}
\begin{aligned}
&\int_{\Om}\<\J(u^{n}_{\ep})\times \De u^{n}_{\ep},\,\J(u^{n}_{\ep})\times J\cdot\n u^{n}_{\ep}\>dx\\
\leq&\de \int_{\Om}|\J(u^{n}_{\ep})\times \De u^{n}_{\ep}|^{2}dx+C(\de)|J|^2_{L^{\infty}(\Om)}\int_{\Om}|\n u^{n}_{\ep}|^{2}dx,
\end{aligned}
\end{equation}
\begin{equation}\label{ineq4}
\begin{aligned}
&\int_{\Om}\<\J(u^{n}_{\ep})\times \De u^{n}_{\ep},\n_u\Phi(\J(u^{n}_{\ep}))\>dx\\
\leq&\de \int_{\Om}|\J(u^{n}_{\ep})\times \De u^{n}_{\ep}|^{2}dx+C(\de)\int_{\Om}|\n_u\Phi(\J(u^{n}_{\ep})|^{2}dx,
\end{aligned}
\end{equation}
\begin{equation}\label{ineq5}
\begin{aligned}
\int_{\Om}\<\J(u^{n}_{\ep})\times \De u^{n}_{\ep}, h_{d}(\J(u^{n}_{\ep}))\>dx
\leq\de \int_{\Om}|\J(u^{n}_{\ep})\times \De u^{n}_{\ep}|^{2}dx+C(\de)\int_{\Om}|u^{n}_{\ep}|^{2}dx,
\end{aligned}
\end{equation}
and
\begin{equation}\label{ineq6}
\begin{aligned}
\int_{\Om}\<\J(u^{n}_{\ep})\times \De u^{n}_{\ep},J\cdot\n u^{n}_{\ep}\>dx
\leq\de \int_{\Om}|\J(u^{n}_{\ep})\times \De u^{n}_{\ep}|^{2}dx+C(\de)|J|^2_{L^{\infty}(\Om)}\int_{\Om}|\n u^{n}_{\ep}|^{2}dx.
\end{aligned}
\end{equation}
In the above argument, we have used the estimate of $h_{d}$ in Lemma \ref{es-h_d}. In view of the above inequalities \eqref{ineq1}-\eqref{ineq6}, by choosing $\de=\frac{\al}{12}$, we can deduce from \eqref{ineq} that
\begin{equation}\label{ineq'}
\begin{aligned}
&\frac{1}{2}\frac{\p }{\p t}\int_{\Om}|\n u^{n}_{\ep}|^{2}dx+\ep\int_{\Om}|\De u^{n}_{\ep}|^{2}dx+\frac{\al}{2}\int_{\Om}|\J(u^{n}_{\ep})\times \De u^{n}_{\ep}|^{2}dx\\
\leq&\, C_{1}(\al)(|\n_{u}\Phi|^2_{L^\infty}+1)Vol(\Om)+C_{2}(\al)\beta^2|J|^{2}_{L^{\infty}(\Om)}\int_{\Om}|\n u^{n}_{\ep}|^{2}dx.
\end{aligned}
\end{equation}
Then, the Gronwall inequality implies that, for any $T>0$ there exists a constant $C(\alpha)$ which depends only on $\Phi$ and $Vol(\Om)$, such that the following estimate holds true
\begin{equation}\label{ineq''}
\begin{aligned}
&\sup_{0\leq t\leq T}\int_{\Om}|\n u^{n}_{\ep}|^{2}dx(t)+\ep\int_{0}^{T}\int_{\Om}|\De u^{n}_{\ep}|^{2}dxdt+\al \int_{0}^{T}\int_{\Om}|\J(u^{n}_{\ep})\times \De u^{n}_{\ep}|^{2}dxdt\\
\leq&\, C(\al)T+(C(\al)I(T)+1)\left(\int_{\Om}|\n u_{0}|^{2}dx+C(\al)T\right)\exp{C(\al)I(T)}.
\end{aligned}
\end{equation}
Here
$$I(T)=\beta^{2}\int_{0}^{T}|J|^{2}_{L^{\infty}(\Om)}dt,$$
and the following fact has been used
$$\int_{\Om}|\n u^{n}_{0}|^{2}dx\leq \int_{\Om}|\n u_{0}|^{2}dx,$$
since
$$\int_{\Om}|\n u^{n}_{0}|^{2}dx=-\int_{\Om}\<\De P_{n}(u_{0}),u_{0}\>dx=\int_{\Om}\<\n P_{n}(u_{0}),\n u_{0}\>dx.$$

On the other hand, choosing test function $w=\frac{\p u^{n}_{\ep}}{\p t}$, we know that there holds
\begin{equation}\label{ineqq}
 \begin{aligned}
&\frac{\ep}{2}\frac{\p }{\p t}\int_{\Om}|\n u^{n}_{\ep}|^{2}dx+\int_{\Om}|\frac{\p u^{n}_{\ep}}{\p t}|^{2}dx\\
=&\int_{\Om}\<\al\J(u^{n}_{\ep})\times \frac{\p u^{n}_{\ep}}{\p t}-\frac{\p u^{n}_{\ep}}{\p t},\,\J(u^{n}_{\ep})\times (\De u^{n}_{\ep}+\n_{u}\Phi(\J(u^{n}_{\ep}))+h_d(u^{n}_{\ep})+\beta J\cdot\n u^{n}_{\ep})\>dx.
\end{aligned}
\end{equation}
By taking direct calculations, we obtain
\begin{equation}\label{ineqq1}
\begin{aligned}
&\int_{\Om}\<\al\J(u^{n}_{\ep})\times \frac{\p u^{n}_{\ep}}{\p t}-\frac{\p u^{n}_{\ep}}{\p t},\,\J(u^{n}_{\ep})\times\n_u\Phi(\J(u^{n}_{\ep}))\>dx\\
\leq&\,\de(1+\al^2)\int_{\Om}\left|\frac{\p u^{n}_{\ep}}{\p t}\right|^{2}dx+C(\de)\int_{\Om}|\n_u\Phi(\J(u^{n}_{\ep})|^{2}dx,
\end{aligned}
\end{equation}
\begin{equation}\label{ineqq2}
\begin{aligned}
&\int_{\Om}\<\al\J(u^{n}_{\ep})\times \frac{\p u^{n}_{\ep}}{\p t}-\frac{\p u^{n}_{\ep}}{\p t},\J(u^{n}_{\ep})\times h_{d}(u^{n}_{\ep})\>dx\\
\leq&\, \de(1+\al^2)\int_{\Om}|\frac{\p u^{n}_{\ep}}{\p t}|^{2}dx+C(\de)\int_{\Om}|u^{n}_{\ep}|^{2}dx,
\end{aligned}
\end{equation}
\begin{equation}\label{ineqq3}
\begin{aligned}
&\int_{\Om}\<\al\J(u^{n}_{\ep})\times \frac{\p u^{n}_{\ep}}{\p t}-\frac{\p u^{n}_{\ep}}{\p t},\J(u^{n}_{\ep})\times J\cdot\n u^{n}_{\ep}\>dx\\
\leq &\, \de(1+\al^2)\int_{\Om}|\frac{\p u^{n}_{\ep}}{\p t}|^{2}dx+C(\de)|J|^2_{L^{\infty}(\Om)}\int_{\Om}|\n u^{n}_{\ep}|^{2}dx,
\end{aligned}
\end{equation}
and
\begin{equation}\label{ineqq4}
\begin{aligned}
&\int_{\Om}\<\al\J(u^{n}_{\ep})\times \frac{\p u^{n}_{\ep}}{\p t}-\frac{\p u^{n}_{\ep}}{\p t},\J(u^{n}_{\ep})\times \De u^{n}_{\ep}\>dx\\
\leq\,& \de(1+\al^2)\int_{\Om}|\frac{\p u^{n}_{\ep}}{\p t}|^{2}dx+C(\de)\int_{\Om}|\J(u^{n}_{\ep})\times \De u^{n}_{\ep}|^{2}dx.
\end{aligned}
\end{equation}
Here we have used the fact
$$\left|\al\J(u^{n}_{\ep})\times \frac{\p u^{n}_{\ep}}{\p t}-\frac{\p u^{n}_{\ep}}{\p t}\right|^{2}\leq (1+\al^{2})\left|\frac{\p u^{n}_{\ep}}{\p t}\right|^{2},$$
since the above two vectors are orthonormal in $\Real^{3}$. Combining the above inequalities \eqref{ineqq1}-\eqref{ineqq4} with \eqref{ineqq} and choosing $\de=\frac{1}{8(1+\al^2)}$ lead to
\begin{equation}\label{ineqq'}
\begin{aligned}
&\frac{\ep}{2}\frac{\p }{\p t}\int_{\Om}|\n u^{n}_{\ep}|^{2}dx+\frac{1}{2}\int_{\Om}|\frac{\p u^{n}_{\ep}}{\p t}|^{2}dx\\
\leq&\, C(\al)+C(\al)\beta|J|^{2}_{L^{\infty}(\Om)}\int_{\Om}|\n u^{n}_{\ep}|^{2}dx+C(\al)\int_{\Om}|\J(u^{n}_{\ep})\times \De u^{n}_{\ep}|^{2}dx.
\end{aligned}
\end{equation}
Hence, in view of \eqref{ineq''}, it follows from the Gronwall inequality that
\begin{equation}\label{ineqq''}
\begin{aligned}
&\ep\sup_{0\leq t\leq T}\int_{\Om}|\n u^{n}_{\ep}|^{2}dx+\int_{0}^{T}\int_{\Om}|\frac{\p u^{n}_{\ep}}{\p t}|^{2}dxdt\\
\leq\,& C(\al)T+(C(\al)I(T)+1)(\int_{\Om}|\n u_{0}|^{2}dx+C(\al)T)\exp{C(\al)I(T)}.
\end{aligned}
\end{equation}
Therefore, we add the inequalities \eqref{ineq''} to \eqref{ineqq''}, then the desired estimates of approximated solution $u^{n}_\ep$ are obtained. Hence, we conclude
\begin{lem}\label{u-es}
Let $u_{0}\in H^{1}(\Om,\U^{2})$. For any $n\in\mathbb{N}$ and $T>0$, there exists a solution $u^{n}_{\ep}\in L^{\infty}([0,T],H^{1}(\Om,\Real^{3}))\cap W^{2,1}_{2}(\Om\times[0,T],\Real^{3})$ to \eqref{app-cur1}. Moreover, there exists a constant $C(\al)$ independent on $u^{n}_\ep$, such that the following a prior estimate holds.
\begin{align*}
&(1+\ep)\sup_{0\leq t\leq T}\norm{u^n_\ep}^2_{H^1(\Om)}+\ep\int_{0}^{T}\int_{\Om}|\De u^{n}_{\ep}|^{2}dxdt+\int_{0}^{T}\int_{\Om}|\frac{\p u^{n}_{\ep}}{\p t}|^{2}dxdt\\
&+\al \int_{0}^{T}\int_{\Om}|\J(u^{n}_{\ep})\times \De u^{n}_{\ep}|^{2}dxdt\leq C(\al)T+(C(\al)I(T)+1)(\norm{u_0}^2_{H^1(\Om)}+C(\al)T)\exp{I(T)}.
\end{align*}
\end{lem}
\subsection{Compactness of the approximated solutions and the limiting map}
In this subsection, we consider the compactness of the approximation solution $u^{n}_{\ep}$ to \eqref{app-cur1} constructed in the above. The main tool to achieve the compactness is the well-known Alaoglu' theorem and the Aubin-Simon compactness theorem (see Lemma \ref{A-S} in Section $2$). Thus, the Lemma \ref{u-es} and Lemma \ref{eq-norm} imply that there exists a subsequence of $\{u^{n}_{\ep}\}$, we still denote it by $\{u^{n}_{\ep}\}$, and a $u_{\ep} \in L^{\infty}([0,T],H^{1}(\Om,\Real^{3}))\cap W^{2,1}_{2}(\Om\times[0,T],\Real^{3})$ such that
$$u^{n}_\ep\rightharpoonup u_\ep,\,\,\text{weakly* in}\, L^{\infty}([0,T],H^{1}(\Om,\Real^{3})),$$
$$u^{n}_\ep\rightharpoonup u_\ep,\,\,\,\text{weakly in}\,W^{2,1}_{2}(\Om\times[0,T],\Real^{3}).$$

Next, let $X=H^{2}(\Om,\Real^{3})$, $B=H^{1}(\Om,\Real^{3})$ and $Y=L^{2}(\Om,\Real^{3})$. Then, Lemma \ref{A-S} tells us
$$u^{n}_{\ep}\rightarrow u_\ep,\,\,\,\text{strongly in}\, L^{p}([0,T], H^{1}(\Om,\Real^{3})),$$
for any $p<\infty$. It follows that $u_{\ep}$ is a strong solution of the equation \eqref{app-cur}.
\begin{lem}\label{sepshmp}
For any fixed $\ep>0$, the limiting map $u_{\ep}$ of $\{u^{n}_{\ep}\}$ is a $W^{2,1}_2$-solution of \eqref{app-cur}. Namely, for any $\varphi\in C^{\infty}(\bar{\Om}\times [0,T])$, there holds
\begin{align*}
&\int_{0}^{T}\int_{\Om}\<\frac{\p u_{\ep}}{\p t},\varphi\>dxdt=\ep\int_{0}^{T}\int_{\Om}\<\De u_{\ep},\varphi\>dxdt\\
&+\al\int_{0}^{T}\int_{\Om}\<\J(u_{\ep})\times h_\beta,\J(u_{\ep})\times\varphi\>dxdt-\int_{0}^{T}\int_{\Om}\<\J(u_{\ep})\times h_\beta,\varphi\>dxdt,	
\end{align*}
where $h_{\beta}(u_{\ep})=\De u_\ep-\n_u\Phi(\J(u_{\ep}))+h_d(u_\ep)+\beta J\cdot\n u_\ep$.
\end{lem}

\begin{proof}
For any $\varphi\in C^{\infty}(\bar{\Om}\times [0,T])$ and $k\in \mathbb{N}$, we set $\varphi_{k}=P_{k}(\varphi)=\sum_{i=1}^{k}g_{i}(t)f_{i}$, where $g_{i}(t)=\<\varphi,f_{i}\>_{L^{2}(\Om)}$. Thus, $\norm{\varphi-\varphi_{k}}_{L^{\infty}([0,T],L^{2}(\Om))}\to 0$ as $k\to \infty$.

We firstly fix $k$ and let $n\geq k$. Since $u^{n}_{\ep}$ is a strong solution of \eqref{app-cur1}, it follows
\begin{align*}
&\int_{0}^{T}\int_{\Om}\<\frac{\p u^{n}_{\ep}}{\p t},\varphi_k\>dxdt=\ep\int_{0}^{T}\int_{\Om}\<\De u^n_{\ep},\varphi_k\>dxdt+\\
&\al\int_{0}^{T}\int_{\Om}\<\J(u^n_{\ep})\times h_\beta(u^{n}_\ep),\J(u^{n}_{\ep})\times\varphi_k\>dxdt-\int_{0}^{T}\int_{\Om}\<\J(u^n_{\ep})\times h_\beta(u^n_\ep),\varphi_k\>dxdt.	 
\end{align*}
The above conclusions on compactness imply
$$\frac{\p u^{n}_{\ep}}{\p t}\rightharpoonup \frac{\p u_{\ep}}{\p t}\quad\text{weakly in}\, L^{2}([0,T],L^{2}(\Om,\Real^{3})),$$
$$\De u^{n}_{\ep}\rightharpoonup\De u_{\ep}\quad\text{weakly in}\,\, L^{2}([0,T],L^{2}(\Om,\Real^{3})),$$
and
$$u^{n}_{\ep}\rightarrow u\quad\text{strongly in}\,\, L^{p}([0,T], H^{1}(\Om,\Real^{3})).$$
It follows
$$h_d(u^{n}_{\ep})\rightarrow h_d(u_{\ep})\quad\text{strongly in}\,\, L^{p}([0,T], H^{1}(\Om,\Real^{3})),$$
$$\J(u^{n}_{\ep})\rightarrow \J(u_{\ep})\quad\text{pointwise a.e}\,\, (x,t)\in \Om\times[0,T],$$
$$\J(u^{n}_{\ep})\times\De u^{n}_{\ep}\rightharpoonup\J(u_{\ep}) \times\De u_{\ep}\quad \text{weakly in} \,\,L^{2}([0,T],L^{2}(\Om)),$$
and
$$\J(u^{n}_{\ep})\times J\cdot \n u^n_\ep \to \J(u_{\ep}) \times J\cdot \n u_{\ep}\quad \text{in} \,\,L^{1}([0,T],L^{2}(\Om)).$$
Here Lemma \ref{es-h_d} and $J\in L^{2}([0,T],L^{\infty}(\Om))$ have been used.

Therefore, using the dominated convergence theorem and the definition of weak convergence, we infer the desired conclusions as $n\to \infty$ and then $k\to \infty$. It remains that we need to check the Neumann boundary condition. Since for any $\xi\in C^{\infty}(\bar{\Om}\times[0,T])$, there holds
 $$\int_{0}^{T}\int_{\Om}\<\De u^{n}_{\ep},\xi\>dxdt=-\int_{0}^{T}\int_{\Om}\<\n u^{n}_{\ep},\n \xi\>dxdt.$$
Let $n\to \infty$, we have
 $$\int_{0}^{T}\int_{\Om}\<\De u_{\ep},\xi\>dxdt=-\int_{0}^{T}\int_{\Om}\<\n u_{\ep},\n \xi\>dxdt,$$
 that is $\frac{\p u_{\ep}}{\p \nu}|_{\p \Om\times [0,T]}=0$.
\end{proof}
By a similar argument with that in \cite{W, JW1, JW2}, we can also establish the following
\begin{lem}\label{mpl}
Let $w:\Om\times[0,T]\to \Real^{3}$ be a solution of the following equation belonging to $W^{1,1}_{2}([0,T]\times \Om; \Real^{3})\cap L^{\infty}([0,T],H^{1}(\Om))$
\begin{equation}
\begin{cases}
\displaystyle\frac{\p w}{\p t}-\ep \De w=\J(w)\times F(u,x,t),&(x,t)\in\Om\times (0,T),\\
w(x,0)=w_{0}(x):\Om\to \U^{2},\\
\displaystyle\frac{\p w}{\p \nu}=0,&(x,t)\in\p\Om\times[0,T],
\end{cases}
\end{equation}
where $F(u,x,t)\in L^{2}(\Om\times [0,T],\Real^{3})$. Then, $|w(x,t)|\leq 1$ for a.e $(x,t)\in \Om\times[0,T]$.
\end{lem}

\begin{proof}
We firstly choose a test function $f=(|w|-1)^{+}\frac{w}{|w|}$. Then, there holds
$$\int_{\Om}\<\frac{\p w}{\p t},f\>dx=-\ep\int_{\Om}\<\n w,\n f\>dx.$$
Then, by a simple computation we can see from the above identity

\begin{equation}\label{Key1'}
\begin{aligned}
&\frac{1}{2}\frac{d}{dt}\int_{|w|> 1}|w|^2\(1-\frac{1}{|w|}\)\,dx+\ep\int_{|w|>1}\frac{|\<\n w,w\>|^2}{|w|^3}\,dx\\
&=\frac{1}{2}\int_{|w|>1}\frac{\<w, w_t\>}{|w|}\,dx-\ep\int_{|w|>1}|\n w|^2\left(1-\frac{1}{|w|}\right)\,dx.
\end{aligned}
\end{equation}
Taking $$g=\frac{(|w|-1)^+}{(|w|-1+\de)}\frac{w}{|w|}$$ as another test function, we get:
\[
\aligned
&\int_{|w|>1}\frac{\<w, w_t\>}{|w|} \frac{|w|-1}{|w|-1+\de}\,dx=-\ep\int_{\Om}\n w\cdot\n\Big[ \frac{w(|w|-1)^+}{|w|(|w|-1+\de)} \Big]\,dx\\
=&-\ep\int_{|w|>1}|\n w|^2\frac{|w|-1}{|w|(|w|-1+\delta)}\,dx\\
&+\ep\int_{|w|>1}\frac{|\<\n w, w\>|^2}{|w|^3}\frac{|w|-1}{|w|+\de-1}\,dx-\ep\int_{|w|>1}\frac{|\<\n w, w\>|^2}{|w|^2}\frac{\de}{(|w|+\de-1)^2}\,dx\\
\leq& -\ep\int_{|w|>1}|\n w|^2\frac{|w|-1}{|w|(|w|-1+\delta)}\,dx+\ep\int_{|w|>1}\frac{|\<\n w, w\>|^2}{|w|^3}\frac{|w|-1}{|w|+\de-1}\,dx.
\endaligned
\]
By the dominated convergence theorem, letting $\de \to0$ we derive from the above
\begin{equation}\label{Key2'}
\int_{|w|>1}\frac{\<w, w_t\>}{|w|}\,dx\leq-\ep\int_{|w|>1}\frac{|\n w|^2}{|w|}\,dx+\ep\int_{|w|>1}\frac{|\<\n w, w\>|^2}{|w|^3}\,dx.
\end{equation}
Combining (\ref{Key1'}) and (\ref{Key2'}) yields
\[
\frac{d}{dt}\int_{|w|>1}|w|^2\Big(1-\frac{1}{|w|}\Big)\,dx\leq 0.
\]
This means that the following function
\[
q(t):=\int_{|w|>1}|w|^2\Big(1-\frac{1}{|w|}\Big)\,dx
\]
is decreasing non-negative function. Noting $|w_0|=1$, i.e. $q(0)=0$, we can see that $q(t)\equiv0$ for any $t>0$. Therefore, we have $|w|\leq1$.
\end{proof}
Using this lemma, we have $|u^{\ep}|\leq 1$. It follows that $\J(u_{\ep})=u_{\ep}$. Thus, $u_{\ep}$ is a $W^{2,1}_{2}$- solution of the following equation
\begin{equation}\label{ep-shmp1}
\begin{cases}
\displaystyle\frac{\p u_{\ep}}{\p t}=\ep \De u_{\ep}-\al u_{\ep}\times(u_{\ep}\times h_\beta(u_{\ep}))-\al u_{\ep}\times h_{\beta}(u_\ep),&(x,t)\in\Om\times\Real^{+},\\
u(x,0)=u_{0}(x), &x\in\Om,\\
\displaystyle\frac{\p u_{\ep}}{\p \nu}=0,&(x,t)\in\p\Om\times\Real^{+}.
\end{cases}
\end{equation}

\subsection{Solutions to Landau-Lifschitz equation with spin-polarized current}\

In the last part of this section, we are in position to prove Theorem \ref{thm1}. Let $u_{\ep}$ be the solution of equation \eqref{app-cur} for $\ep>0$, constructed in the above. We will show that there exists a subsequence of $\{u_{\ep}\}$, and still denote it by $\{u_{\ep}\}$, converges to a map $u\in L^{\infty}([0,T],H^{1}(\Om,\U^{2}))\cap W^{1,1}_{2}(\Om\times[0,T],\U^{2})$, which is a weak solution of \eqref{cur1}. Moreover, there is a constant $C(\al)$, such that there holds true
\begin{align*}
\sup_{0\leq t\leq T}\norm{u}^2_{H^1(\Om)}+\int_{0}^{T}\int_{\Om}|\frac{\p u}{\p t}|^{2}dxdt\leq C(\al)T+(C(\al)I(T)+1)(\norm{u_0}^2_{H^1(\Om)}+C(\al)T)\exp{I(T)}.
\end{align*}
\medskip

\begin{proof}[\textbf{The proof of Theorem \ref{thm1}}]

We divide the proof into two steps.\\

\noindent\emph{Step 1: The compactness of $\{u_{\ep}\}$ and the limiting map.}
\medskip

From the above arguments, we know that, for any $T>0$, $u_{\ep} \in L^{\infty}([0,T],H^{1}(\Om,\Real^{3}))\cap W^{2,1}_{2}(\Om\times[0,T],\Real^{3})$. Furthermore, the estimates of $u^{n}_{\ep}$ in Lemma \ref{u-es} and lower semi-continuity imply that $u_{\ep}$ satisfies the following uniform estimate
 \begin{align*}
 &(1+\ep)\sup_{[0,T]}\norm{u_{\ep}}^{2}_{H^1(\Om)}+\al \int_{0}^{T}\int_{\Om}|u_{\ep}\times \De u_{\ep}|^{2}dx+\int_{0}^{T}\int_{\Om}|\frac{\p  u_{\ep}}{\p t}|^{2}dx+\ep\int_{0}^{T}\int_{\Om}|\n^{2} u_{\ep}|^{2}dx\\
\leq &C(\al)T+(C(\al)I(T)+1)(\norm{u_0}^2_{H^1(\Om)}+C(\al)T)\exp{I(T)}.
 \end{align*}
Then, there exists a $u\in L^{\infty}([0,T],H^{1}(\Om,\Real^{3}))\cap W^{1,1}_{2}(\Om\times[0,T],\Real^{3})$ such that
$$u_{\ep}\rightharpoonup u,\quad\text{weakly* in}\, L^{\infty}([0,T],H^{1}(\Om,\Real^{3}))$$
and
$$u_{\ep}\rightharpoonup u,\quad\text{weakly in}\,W^{1,1}_{2}(\Om\times[0,T],\Real^{3})$$
as $\ep\to 0$.

By letting $X=H^{1}(\Om,\Real^{3})$ and $B=Y=L^{2}(\Om,\Real^{3})$ in Lemma \ref{A-S}, we have
$$u_{\ep}\rightarrow u,\quad\text{strongly in}\, L^{p}([0,T], L^{2}(\Om,\Real^{3})).$$
for any $1\leq p<\infty$.
Moreover, we have
$$h_d(u_{\ep})\rightarrow h_d(u),\quad\text{strongly in}\, L^{p}([0,T], L^{2}(\Om,\Real^{3}))$$
and
$$u_{\ep}\rightarrow u,\quad\text{point-wise a.e}\,(x,t)\in\Om\times[0,T].$$
The above convergence implies
$$u_{\ep}\times h_d(u_\ep)\to u\times h_d(u),\quad\text{in}\,L^{2}(\Om\times[0,T],\Real^{3}),$$
$$u_{\ep}\times \n u_{\ep}\rightharpoonup u\times \n u,\quad\text{weakly in}\,L^{2}(\Om\times[0,T],\Real^{3}),$$
$$u_{\ep}\times \n_u\Phi(u_\ep)\rightharpoonup u\times \n_u\Phi(u),\quad\text{weakly in}\,L^{2}(\Om\times[0,T],\Real^{3}),$$
and
$$u_{\ep}\times J\cdot\n u_\ep \to u\times J\cdot\n u,\quad\text{in}\,L^{1}(\Om\times[0,T],\Real^{3}).$$
Here we have used $|u_\ep|\leq 1$ and $J\in L^{2}( [0,T],L^\infty(\Om))$.

By a similar argument with that in Lemma \ref{sepshmp}, we can show
$$\frac{\p u}{\p \nu}|_{\p\Om\times [0,T]}=0.$$
Let $u_{\ep}$ be a test function for \eqref{app-cur}, we have
$$\int_{\Om}|u_{\ep}|^{2}dx+2\ep\int_{0}^{T}\int_{\Om}|\n u_{\ep}|^{2}dxdt=Vol(\Om),$$
As $\ep\to 0$, there holds
$$\int_{\Om}(|u|^{2}-1)dx=0,$$
it implies $|u|=1$, since $|u|\leq 1$.
\medskip

\noindent\emph{Step 2: Weak solutions of \eqref{cur1}.}
\medskip

Let $\varphi \in C^{\infty}(\bar{\Om}\times[0,T])$. Then Lemma \ref{sepshmp} tells us that
\begin{align*}
&\int_{0}^{T}\int_{\Om}\<\frac{\p u_{\ep}}{\p t},\varphi\>dxdt=-\ep\int_{0}^{T}\int_{\Om}\<\De u_{\ep},\varphi\>dxdt\\
&+\al\int_{0}^{T}\int_{\Om}\<u_{\ep}\times h_\beta(u_\ep) ,u_{\ep}\times\varphi\>dxdt-\int_{0}^{T}\int_{\Om}\<u_{\ep}\times h_\beta(u_\ep) ,\varphi\>dxdt.
\end{align*}
By the convergence obtained in $Step\,1$, we have
$$LHS\to \int_{0}^{T}\int_{\Om}\<\frac{\p u}{\p t},\varphi\>dxdt, \quad\quad\text{as}\,\,\ep \to \text{0}.$$
For the right hand side, we take integrating by parts to obtain
 \begin{align*}
RHS=&\ep\int_{0}^{T}\int_{\Om}\<\n u_{\ep}, \n \varphi\>dxdt+\al\int_{0}^{T}\int_{\Om}\<u_{\ep}\times\De u_{\ep},u_{\ep}\times\varphi\>dxdt-\int_{0}^{T}\int_{\Om}\<u_{\ep}\times\De u_{\ep},\varphi\>dxdt\\
&+\al\int_{0}^{T}\int_{\Om}\<u_{\ep}\times(\n_u\Phi(u_\ep)+h_d(u_\ep)+\beta J\cdot\n u_\ep),u_{\ep}\times\varphi\>dxdt\\
&-\int_{0}^{T}\int_{\Om}\<u_{\ep}\times(\n_u\Phi(u_\ep)+h_d(u_\ep)+\beta J\cdot\n u_\ep),\varphi\>dxdt\\
=&\ep\int_{0}^{T}\int_{\Om}\<\n u_{\ep}, \n \varphi\>dxdt+I+II+III+IV.
\end{align*}
Since $\n u_{\ep}\rightharpoonup\n u\,\text{weakly in}\, L^{2}([0,T],L^{2}(\Om,\Real^{3}))$, we have
$$\ep\int_{0}^{T}\int_{\Om}\<\n u_{\ep}, \n \varphi\>dxdt\to 0,$$
as $\ep \to 0$.

For the third term $III$ and $IV$, by the arguments in $Step\, 1$ we can get directly
\begin{align*}
&III=\al\int_{0}^{T}\int_{\Om}\<u_{\ep}\times(\n_u\Phi(u_\ep)+h_d(u_\ep)+\beta J\cdot\n u_\ep),u_{\ep}\times\varphi\>dxdt\\
&\to \al\int_{0}^{T}\int_{\Om}\<u\times(\n_u\Phi(u)+h_d(u)+\beta J\cdot\n u),u\times\varphi\>dxdt,
\end{align*}
\begin{align*}
&IV=-\int_{0}^{T}\int_{\Om}\<u_{\ep}\times(\n_u\Phi(u_\ep)+h_d(u_\ep)+\beta J\cdot\n u_\ep),\varphi\>dxdt\\
&\to-\int_{0}^{T}\int_{\Om}\<u\times(\n_u\Phi(u)+h_d(u)+\beta J\cdot\n u),\varphi\>dxdt.
\end{align*}

Next, we show the convergence of term $I$, so as $II$ by almost the same argument as in the above.
\begin{align*}
&I=\al\int_{0}^{T}\int_{\Om}\<u_{\ep}\times \De u_{\ep},u_\ep\times\varphi\>dxdt\\
&=\al\int_{0}^{T}\int_{\Om}\<u_{\ep}\times \De u_{\ep},u\times\varphi\>dxdt+\al\int_{0}^{T}\int_{\Om}\<u_{\ep}\times \De u_{\ep},(u_\ep-u)\times\varphi\>dxdt\\
&=\al(I'+I'').
\end{align*}
For $I''$, we have
\begin{align*}
|I''|\leq |\varphi|_{L^{\infty}}(\sup_{\ep>0}\int_{0}^{T}\int_{\Om}|u_{\ep}\times\De u_{\ep}|^{2}dxdt)^{1/2}(\int_{0}^{T}\int_{\Om}| u_{\ep}-u|^{2}dxdt)^{1/2}\to 0
\end{align*}
as $\ep \to 0$.
It remains to deal with the first term $I'$. Using integration by parts, we have
\begin{align*}
I'&=-\int_{0}^{T}\int_{\Om}\<u_{\ep}\times \n u_{\ep},\n(u\times\varphi)\>dxdt\\
&\to -\int_{0}^{T}\int_{\Om}\<u\times \n u,\n(u\times\varphi)\>dxdt,\\
\end{align*}
since
$$u_{\ep}\times \n u_{\ep}\rightharpoonup u\times \n u,\text{weak in}\,L^{2}(\Om\times[0,T],\Real^{3}).$$
By a similar argument to the above, we also have
$$II\to \int_{0}^{T}\int_{\Om}\<u\times \n u,\n\varphi\>dxdt.$$

To summarize the above arguments, we conclude that the limiting map $u$ satisfies the following equation
\begin{equation}\label{wcur}
\begin{aligned}
&\int_{0}^{T}\int_{\Om}\<\frac{\p u}{\p t},\varphi\>dxdt=-\al\int_{0}^{T}\int_{\Om}\<u\times \n u,\n(u\times\varphi)\>dxdt+\int_{0}^{T}\int_{\Om}\<u\times \n u,\n \varphi\>dxdt\\
&+\al\int_{0}^{T}\int_{\Om}\<u\times \tilde{h}_\beta(u) ,u\times\varphi\>dxdt-\int_{0}^{T}\int_{\Om}\<u\times \tilde{h}_\beta(u) ,\varphi\>dxdt,
\end{aligned}
\end{equation}
where $\tilde{h}_\beta(u)=-\n_u\Phi(u)+h_d(u)+\beta J\cdot\n u$ and $\varphi\in C^{\infty}(\bar{\Om}\times [0,T],\Real^{3})$.

Let $u\times \varphi$ replace $\varphi$ in the above equation, we have
 \begin{equation}\label{wcur1}
 \begin{aligned}
 &\int_{0}^{T}\int_{\Om}\<u\times\frac{\p u}{\p t},\varphi\>dxdt=-\al\int_{0}^{T}\int_{\Om}\<u\times \n u,\n \varphi\>dxdt-\int_{0}^{T}\int_{\Om}\<u\times \n u,u\times \varphi\>dxdt\\
 &+\al\int_{0}^{T}\int_{\Om}\<u\times \tilde{h}_\beta(u) ,\varphi\>dxdt+\int_{0}^{T}\int_{\Om}\<u\times \tilde{h}_\beta(u) ,u\times \varphi\>dxdt.
 \end{aligned}
 \end{equation}

To add \eqref{wcur} and $-\al\times$\eqref{wcur1}, we have
\begin{equation}\label{wcur3}
\begin{aligned}
&\int_{0}^{T}\int_{\Om}\<\frac{\p u}{\p t},\varphi\>dxdt-\al\int_{0}^{T}\int_{\Om}\<u\times\frac{\p u}{\p t},\varphi\>dxdt\\
=&(1+\al^2)\int_{0}^{T}\int_{\Om}\left(\<u\times \n u,\n \varphi\>-\<u\times\tilde{h}_\beta(u) ,\varphi\>\right)dxdt.\\
\end{aligned}
\end{equation}
Thus, the proof finished.
\end{proof}

\section{The heat flows associate to micromagnetic energy functional}\label{s: hf-m}
In the present section, we intend to show the well-posedness of global solutions to equations \eqref{hf-m} and \eqref{hf-m1} respectively. We will only give the sketches of the proofs for Theorems \ref{thm2} and \ref{thm3}, since the arguments go almost the same as that of Theorem \ref{thm1}. Let $\Om$ be a smoothly bounded  domain in $\Real^{m}$, where $m\geq 3$. Recall that the equation for us to consider is as following
\begin{equation}\label{hf-m2}
\left\{
\begin{aligned}
& u_t=-u\times(u\times \bar{h}_\beta), \quad (x, t)\in \Om\times (0,T),\\
& u(0)=u_0: \Om\to \U^{2},\quad \frac{\partial u}{\partial\nu}=0\quad \mbox{on}\,\, \p\Om\times[0,T].
\end{aligned}\right.
\end{equation}
As before, we also consider the following approximate equation
\begin{equation}\label{app-hf}
\left\{
\begin{array}{llll}
\aligned
&u_t=\ep\De u-\J(u)\times(\J(u)\times \bar{h}_\beta(u)), \quad \mbox{in}\,\, \Omega\times(0,T),\\
&u(0,\cdot)=u_0:\,\Omega\rightarrow \U^2,\quad\frac{\partial u}{\partial\nu}=0\quad \mbox{on}\,\, \partial\Omega,
\endaligned
\end{array}
\right.
\end{equation}
where $\mathfrak{J}(u)$ and the extension $\tilde{\Phi}$ of $\Phi$ are defined by the same way as in Section $3$, and
$$\bar{h}_\beta(u)=\De u-\n_u\Phi(\J(u))+\beta J\cdot\n u.$$
For simplicity, we still denote $\tilde{\Phi}$ by $\Phi$.

By Galerkin method, it is not difficult to prove that above equation admits a solution $u_\ep$ in $L^{\infty}([0,T],H^{1}(\Om,\Real^{3}))\cap W^{2,1}_{2}(\Om\times[0,T],\Real^{3})$ for any $T>0$, which satisfies the following uniform estimates with respect to $\ep$
 \begin{align*}
&(1+\ep)\sup_{[0,T]}\int_{\Om}|\n u_{\ep}|^{2}dx+\int_{0}^{T}\int_{\Om}|\J(u_{\ep})\times \De u_{\ep}|^{2}dx+\int_{0}^{T}\int_{\Om}|\frac{\p  u_{\ep}}{\p t}|^{2}dx+\ep\int_{0}^{T}\int_{\Om}|\n^{2} u_{\ep}|^{2}dx\\
\leq&C(\al)T+(C(\al)I(T)+1)\left(\int_{\Om}|\n u_{0}|^{2}dx+C(\al)T\right)\exp{I(T)}.
\end{align*}
Its proof is similar to that in the previous section. Using Lemma \ref{mpl}, we also have $|u_\ep|\leq 1$, then $\J(u_\ep)=u_\ep$.  Now, we are in the position to prove Theorem \ref{thm2}.

\begin{proof}[\textbf{The proof of Theorem \ref{thm2}}]	\
	
By almost the same arguments as in the proof of Theorem \ref{thm1}, we can get a map $$u\in L^{\infty}([0,T],H^{1}(\Om,\U^{2}))\cap W^{1,1}_{2}(\Om\times[0,T],\U^{2}),$$ which is a limiting map of $\{u_\ep\}$ by let $\ep\to 0$. Moreover, It satisfies the following
\begin{equation}\label{whf-c}
\begin{aligned}
&\int_{0}^{T}\int_{\Om}\<\frac{\p u}{\p t},\varphi\>dxdt=-\int_{0}^{T}\int_{\Om}\<u\times \n u,\n(u\times\varphi)\>dxdt\\
&+\int_{0}^{T}\int_{\Om}\<u\times(-\n_u\Phi(u)+\beta J\cdot\n u),u\times\varphi\>dxdt
\end{aligned}
\end{equation}
where $\varphi\in C^{\infty}(\bar{\Om}\times [0,T],\Real^{3})$.

For the first term in the right hand side of above equation, there holds
\begin{align*}
&-\int_{0}^{T}\int_{\Om}\<u\times\n u,\n (u\times\varphi\>dxdt\\
=&-\int_{0}^{T}\int_{\Om}\<u\times\n u,\n u\times(\varphi-\varphi\cdot u u)\>dxdt-\int_{0}^{T}\int_{\Om}\<u\times\n u,u\times\n(\varphi-\varphi\cdot u u)\>dxdt\\
=&-\int_{0}^{T}\int_{\Om}\<u\times\n u,u\times\n\varphi\>dxdt+\int_{0}^{T}\int_{\Om}\<u\times\n u,u\times(\varphi\cdot u) \n u)\>dxdt\\
=&-\int_{0}^{T}\int_{\Om}\<\n u,\n \varphi\>dxdt+\int_{0}^{T}\int_{\Om}|\n u|^{2}\<\varphi, u\>dxdt,
\end{align*}
where we have used $\n u\times(\varphi-\varphi\cdot u u)=0$ or $c(x)u$ for some $1$-form $c(x)$.
Thus, $u$ is a weak solution of equation \eqref{hf-m}.
\end{proof}
By almost the same process showed in the above, we can also get the existence of the weak solution to heat flow \eqref{hf-m1}. Thus, the proof of Theorem \ref{thm3} is completed.

\medskip
\noindent {\it\textbf{Acknowledgements}}: The authors are supported by NSFC grant (No.11731001).

\end{document}